\documentclass[12pt, a4paper]{elsarticle}
\usepackage{amsmath, amsthm, amscd, amsfonts, amssymb, graphicx, color}
\usepackage[bookmarksnumbered,colorlinks, plainpages]{hyperref}

\newtheorem{theorem}{Theorem}[section]
\newtheorem{lemma}[theorem]{Lemma}
\newtheorem{proposition}[theorem]{Proposition}
\newtheorem{corollary}[theorem]{Corollary}
\theoremstyle{definition}

\newtheorem{example}[theorem]{Example}

\newtheorem{conjecture}[theorem]{Conjecture}
\theoremstyle{remark}

\numberwithin{equation}{section}
\makeatletter
\newcommand{\vast}{\bBigg@{3}}
\newcommand{\Vast}{\bBigg@{4}}
\makeatother

\begin{document}

\title{\begin{Huge}\textbf{Self-Contained Graphs}\end{Huge}}
\author{\textbf{Mohammad Hadi Shekarriz}\footnote{\small{Corresponding author\\ \indent \textit{E-mail addresses.} mh.shekarriz@stu.um.ac.ir and mirzavaziri@um.ac.ir}} \textbf{and Madjid Mirzavaziri}}
\address{Department of Pure Mathematics, Ferdowsi University of Mashhad,\\
P. O. Box 1159, Mashhad 91775, Iran.}

\begin{abstract}
A self-contained graph is an infinite graph which is isomorphic to one of its proper induced subgraphs. In this paper, these graphs are studied by presenting some examples and defining some of their sub-structures such as removable subgraphs and the foundation. Then, we show that the general version of graph alternative conjecture, which says every graph has infinitely many strong twins or none, can be deduced from its connected version, which says every connected graph has infinitely many connected strong twins or none. Moreover, we try to find out under what conditions on two arbitrary removable subgraphs, their union is also a removable subgraph.\\


\end{abstract}

\maketitle

\section{Introduction}
Although it is impossible for finite graphs, there are infinite graphs which are isomorphic to one of their proper subgraphs. Infinite empty graphs (which are graphs with no edges), rays, infinite stars and complete infinite graphs are trivial examples, but there are more interesting ones such as the Rado (random, unique ultra-homogeneous
countable) graph. In this paper, our intention is to study properties and structures of \emph{self-contained graphs} which are isomorphic to one of their proper \textit{induced} subgraphs.

The problem of considering an isomorphism of a graph to one of its proper subgraphs has its origin at least to 1970s, for example, they have appeared in Halin's work \cite{Halin1}. But self-contained graphs have only recently fascinated mathematicians by the so-called \emph{``Graph alternative conjecture''} which has started in \cite{Bonato1} where Bonato and Tardif studied \emph{twins} of infinite graphs under the phrase \emph{``mutually embeddable graphs''}; two non-isomorphic graphs $G$ and $H$ are called \textit{``strong twins''} if $G$ is isomorphic to a proper induced subgraph of $H$ and $H$ is also isomorphic to a proper induced subgraph of $G$. They asked a question that if $G$ and $H$ are twins, then do $G$ and $H$ belong to an infinite family of twins? A few years later, they extended their study of twins in \cite{Bonato2} where they noted that if an infinite graph has a strong twin, then it is isomorphic to one of its proper induced subgraphs, i.e., in our phrase, every graph that has a strong twin is also self-contained. They also conjectured that every infinite tree has either infinitely many tree-twins or none. They called it \emph{``the tree alternative conjecture''} and proved it for all rayless trees \cite{Bonato2}.

In 2009, Tyomkyn proved that the tree alternative conjecture is true for all rooted trees \cite{Tyomkyn}. Moreover, using Schmidt's method (which has been introduced in English by Halin in \cite{Halin2}), Bonato et al. proved that (i) a rayless graph has either infinitely many twins or none, and (ii) a connected rayless graph has either infinitely many connected twins or none \cite{rayless}. Therefore, we have the following two versions of graph alternative conjecture for strong twins:

\begin{conjecture}
\label{general}
A graph $G$ has infinitely many strong twins or none.
\end{conjecture}

\begin{conjecture}
\label{connected}
A connected graph $G$ has infinitely many connected strong twins or none.
\end{conjecture}

Unfortunately, there are very few further developments to the graph alternative conjecture. We can only name the recent unpublished work of Laflamme, Pouzet and Sauer \cite{pouzet}, in which it is proved that tree alternative conjecture is also true for \emph{scattered trees}, that are trees containing no subdivision of the complete binary tree as a subtree. 

Here in this paper, we start studying self-contained graphs by presenting some examples and recognizing some of their internal structures such as removable subgraphs, the foundation, torsion of a removable subgraph, etc. Afterwards, in Section 3, we prove the following theorem.

\begin{theorem}
\label{main}
Conjecture \ref{connected} implies Conjecture \ref{general}.
\end{theorem}

In the last section, we try to find out under what conditions on two removable subgraph of a self-contained graph, their union is also a removable subgraph. Solving this could be a big step toward solving the general case of the graph alternative conjecture, because it may enable us to inductively construct infinitely many twins for a graph that we know it already has a twin.

To read this paper, besides a general knowledge of graph theory and infinite sets, only few definitions of infinite graph theory is needed, all of which can be found in Section 8 of \cite{Diestel}.

To simplify, we use the notation $\emptyset$ for the null graph, the unique graph that has no vertices. Furthermore, we use the notations $\subseteq$ and $\subset$ for the subgraph relations, and $G \setminus H$ for the induced subgraph $G[V(G) \setminus V(H)]$.

\section{Self-contained graphs: notations, structures and examples}
A self-contained graph can be a digraph or a multigraph, but we limit our attention to simple graphs. These graphs have interesting properties we are intended to study here. One of many interesting property of self-contained graphs is that they are essential to all connected infinite graphs. By theorem 8.2.1 of \cite{Diestel}, every connected infinite graph contains a ray or a vertex of infinite degree (an infinite star), so, every connected infinite graph contains a self-contained subgraph (not necessarily an induced one).

Let $G$ be a self-contained graph. A non-empty proper subgraph $H$ of $G$ is a \emph{removable subgraph} of $G$ if $G \setminus H \cong G$. Then we write $H \in \mathrm{Rem}(G)$ and by $\mathrm{Iso}_{G} (H)$ we mean the set of all isomorphisms $f:G\longrightarrow G \setminus H$. It can be easily seen (at the sight of Proposition \ref{union-rem}) that $\vert \mathrm{Rem}(G) \vert \geq \aleph_{0}$, where $\aleph_0$ is the least infinite cardinal, equipotent to the cardinality of natural numbers. The \emph{linkage of $H$}, $\mathrm{link}(H)$, is the edges joining $H$ to $G \setminus H$, i. e.,  $\mathrm{link}(H) = E(H,G\setminus H)=E(G) \setminus \Big(E(H) \cup E\big( G \setminus H\big) \Big)$.

\begin{proposition}
\label{7}
Let $G$ be a self-contained graph and $G^c$ be its complement. Then $G^c$ is a self-contained graph. Moreover, if $H \in \mathrm{Rem}(G)$ then $H^{c} \in \mathrm{Rem}(G^{c})$.
\end{proposition}
\begin{proof} Let $H$ be proper subgraph of $G$ such that $V(H) \neq \emptyset$ and $G \cong G \setminus H$. Then $G^{c} \cong (G \setminus H)^{c} = G^{c} \setminus H^{c}$ which shows that $G^c$ is a self-contained graph and $H^c$ is a removable subgraph of it.
\end{proof}
\begin{example}
Let $R$ be the random graph (which is also known as the Rado graph). Then $R$ is a self-contained graph that is also self-complimentary, i.e., $R \cong R^{c}$. For the random graph and its interesting properties see \cite{Rado} or Section 8.3 of \cite{Diestel}.
\end{example}
The following two propositions are easy to prove. Meanwhile, Proposition \ref{8} has been stated for locally finite graphs in \cite{Halin1} (Proposition 11, page 263).
\begin{proposition}
\label{union-rem}
Let $G$ be a self-contained graph, $P \in \mathrm{Rem}(G)$ and $Q$ be an induced subgraph of $G \setminus P$. Then $Q \in \mathrm{Rem}(G \setminus P)$ if and only if $P \cup Q \in \mathrm{Rem}(G)$.
\end{proposition}

\begin{proposition}
\label{8} Let $G$ be a self-contained graph and $H \in \mathrm{Rem}(G)$. Then $G$ contains infinitely many vertex disjoint copies of $H$. 
\end{proposition}

The following theorem characterizes removable subgraphs of disconnected self-contained graphs. It worth mentioning that, because we do not need subgraphs to be induced, it is true for all graphs with self-embedding, not only for self-contained graphs.

\begin{theorem}
\label{dis-conn}
Let $G$ be a disconnected self-contained graph. Then at least one of the following statements is true:
\begin{itemize}
\item[i.] $G$ contains a connected component which is also a self-contained graph.
\item[ii.] There exists a connected component of $G$ which is a removable subgraph of it.
\item[iii.] For each $H \in \mathrm{Rem}(G)$ and $f \in \mathrm{Iso}_{G}(H)$ there exist infinitely many connected components of $G$, namely $\{ G_{z} \}_{z \in \mathbb{Z}}$, such that for each $z \in \mathbb{Z}$ we have $H$ and $f(H)$ have non-empty intersections with $G_z$.
\end{itemize}
\end{theorem}
\begin{proof}
Let the first statement be false about $G$, $H \in \mathrm{Rem}(G)$ and $f: G \longrightarrow G \setminus H$ be an isomorphism. Then there are three possibilities:
\begin{itemize}
\item[Case 1.] $H$ contains a connected component $P$ of $G$ as a proper induced subgraph. If $G$ has infinitely many connected components isomorphic to $P$, then suppose $R= \{ P_{0} = P, P_{1}, P_{2}, \ldots \}$ be a subfamily of copies of $P$ in $G$ and define $g: G \longrightarrow G \setminus P$ so that it moves $P_{i}$ to $P_{i+1}$ for each $i=0,1,2, \ldots$ and fixes all other vertices outside $R$. Since $g$ is an isomorphism, $P$ is a removable subgraph of $G$. Therefore, it suffice to consider the case that $G$ has only finitely many connected component isomorphic to $P$. Suppose that $m\geq n>0$ are the number of connected components of isomorphic to $P$ in $G$ and $H$ respectively. Then, there must be a connected component of $G$, namely $\hat{G}$, such that $\hat{G} \cap H \neq \emptyset$ and $\hat{G} \setminus H$ contains at least one connected component, namely $Q$, which is isomorphic to $P$. Therefore, $(H \setminus P) \cup Q)$ is another removable subgraph of $G$ which has $n-1$ connected components isomorphic to $P$. By iterating this process at most to the order of power set of all connected components of $G$ which are also subsets of $H$, we find a removable subgraph $\hat{H}$ that does not contain any connected component of $G$. Then, we can consider $\hat{H}$ instead of $H$, and other cases might happen for it.

\item[Case 2.] $H$ does not contain a connected component of $G$, but it is made of parts of finitely many connected components of $G$, namely $G_{1}, G_{2}, \ldots , G_{k}$. If $f \in \mathrm{Iso}_{G}(H)$, then $f(G_{i})$ must be a connected component of $G \setminus H$, for $i=1,2, \ldots ,k$. So, one of the two possibilities may occur:
\begin{itemize}
\item[Case 2.1.] There is $1 \leq i \leq k$ such that $f^{n}(H \cap G_{i}) \nsubseteq G_{j}$ for all $j=1,2,\ldots ,k$ and $n \in \mathbb{N}$. Then $G_{i}, f^{1}(G_{i}), f^{2}(G_{i}), \ldots$ are infinitely many isomorphic copies of the connected component $G_i$, and hence the second statement is true.
\item[Case 2.2.] For all $1 \leq i \leq k$ and $n \in \mathbb{N}$ we have $f^{n}(H \cap G_{i}) \subset G_{j}$ for some $j=1,2,\ldots ,k$. Then $G_i$ must contain one of $f^{1}(G_{i}), f^{2}(G_{i}), \ldots$, which means that $G_{i}$ is a self-contained component of $G$, a contradiction to our assumption.
\end{itemize}
\item[Case 3.] $H$ is made of parts of infinitely many connected components of $G$, and so there must be a countable subfamily of them that could be named $\{ G_{z} \}_{z \in \mathbb{Z}}$. If $f \in \mathrm{Iso}_{G}(H)$, then $f(G_{z})$ must be a connected component of $G \setminus H$, for each $z\in \mathbb{Z}$. So, one of the two possibilities may occur:
\begin{itemize}
\item[Case 3.1.] There is $z\in \mathbb{Z}$ such that $f^{n}(H \cap G_{z}) \nsubseteq G_{t}$ for all $t\in \mathbb{Z}$ and $n \in \mathbb{N}$. Then $G_{z}, f^{1}(G_{z}), f^{2}(G_{z}), \ldots$ are infinitely many isomorphic copies of the connected components $G_z$ which are also connected components of $G$, and hence the second statement is true again.
\item[Case 3.2.] For all $z \in \mathbb{Z}$ and $n \in \mathbb{N}$ we have $f^{n}(H \cap G_{z}) \subset G_{t}$ for some $t\in \mathbb{Z}$. The case $t=z$ is absurd since it means that $G_z$ is self-contained, so we suppose $t\neq z$. Then, for each $z \in \mathbb{Z}$ we have $H$ and $f(H)$ have non-empty intersections with $G_z$, which means that the third statement is true.
\end{itemize}
\end{itemize}
Since these cases cover all possibilities and in each case the second or the third statement is true, we are done with the proof.
\end{proof}

Self-contained graphs may have substructures that are not removable. Let us see an example first:
\begin{example}
\label{9}
Let $G$ be the graph with vertex set $V = \{ 0, 1, 2, \ldots \}$ and its edges defined as follows: $0$ is adjacent to all other vertices while other vertices are only adjacent to $0$ and their consecutive vertices. Consequently, vertex degrees of $0$, $1$ and $i$ for $i \geq 2$ are $\infty$, $2$ and $3$, respectively. It is obvious that $G$ is a self contained graph since there is an isomorphism $G \cong G \setminus \{ 1 \}$. The vertex $0$ is not a vertex of a removable subgraph of $G$.
\end{example}

Let $G$ be a self-contained graph. A subgraph $H$ of $G$ is called an \emph{asset} to $G$ if the intersection of $V(H)$ and vertex set of every removable subgraph of $G$ is empty. The union of all assets of a self-contained graph $G$, is an asset to $G$ which is called \emph{the foundation} of $G$ and shown by $\mathrm{Fnd}(G)$. It is easy to verify that $$\mathrm{Fnd}(G) = \bigcap_{H \in \mathrm{Rem}(G)} G \setminus H.$$

\begin{example}
\label{11} We have seen in Example \ref{9} that $\mathrm{Fnd}(G)$ can be a finite induced subgraph of $G$. Here are two examples of self-contained graphs, one of which has an infinite foundation and the other has a null one.
\begin{itemize}
\item[(a)] Let $G$ be the graph whose vertex set consists of three disjoint copies of natural numbers: $A_{n} = \{ a_{n1}, a_{n2}, \ldots \}$, 
for $n=1, 2, 3$. The edges of $G$ is of the following three kinds:
\begin{itemize}
\item[i.] $\{ a_{1j} a_{1(j+1)} \} \in E(G)$ for all $j \in \mathbb{N}$,
\item[ii.] $\{ a_{11} a_{2j} \} \in E(G)$ for all $j \in \mathbb{N}$,
\item[iii.] $\{ a_{14} a_{3j} \} \in E(G)$ for all $j \in \mathbb{N}$.
\end{itemize} 
Then $G$ is a self-contained graph and $V(\mathrm{Fnd}(G))  = A_{1}$. Hence, $\mathrm{Fnd}(G)$ is a ray.
\item[(b)] The $\mathbb{N} \times \mathbb{N}$ grid, the graph on $\mathbb{N}^2$, in which two vertices $(m,n)$ and $(m^{'},n^{'})$ are adjacent if and only if $\vert m - m^{'} \vert + \vert n - n^{'}\vert = 1$, is a self-contained graph whose foundation is the null graph.
\end{itemize}
\end{example}
As we have seen in Example \ref{11} (b), the foundation of a self-contained graph may have no vertices. A more simple example is a ray, whose foundation is also the null graph. Meanwhile, when $G$ is a self-contained graph with non-empty foundation, every isomorphism of $G$ onto one of its proper induced subgraphs, $H$, maps the foundation of $G$ onto the foundation of $H$. Although $\mathrm{Fnd}(H)$ is again a proper induced subgraph of $G$, it can be different from $\mathrm{Fnd}(G)$ (see Example \ref{15}). By the way, we have the following proposition.
\begin{proposition} 
\label{12}
Let $G$ be a self-contained graph and $H \in \mathrm{Rem}(G)$. Then we have $\mathrm{Fnd}(G) \subseteq \mathrm{Fnd}(G\setminus H)$. 
\end{proposition}
\begin{proof}
Suppose $v \in V \big( \mathrm{Fnd}(G) \big)$ but $v \notin V \big( \mathrm{Fnd}(G \setminus H) \big)$. Hence, there is $P \in \mathrm{Rem}(G \setminus H )$ such that $v \in V(P)$. Now, by Proposition \ref{union-rem} we have $H \cup P \in \mathrm{Rem}(G)$ and $v \in V(H \cup P)$, which is a contradiction since $v$ is a vertex of the foundation of $G$.
\end{proof}

We end this section with following theorem and its corollaries. Meanwhile, if $P=x_{0}x_{2} \ldots x_{k}$ is a path from $x_{0}$ to $x_{k}$, then $x_{i}Px_{j}$ represents the subpath $x_{i}x_{i+1} \ldots x_{j}$ of $P$.

\begin{theorem}
\label{ray}
Let $G$ be a connected self-contained graph. If $\mathrm{Fnd}(G)=\emptyset$ then $G$ contains a ray.
\end{theorem}
\begin{proof}
We inductively construct a ray in $G$. In each step, we add a finite path to a path we have already chosen and make sure that vertex sets of these two paths have only a singleton as their intersection; the last vertex of the first one coincides with the first vertex of the one we want to choose. Let $v_{1} \in V(G)$ and, put $P_{0}$ be the path consisting of the single vertex $v_{1}$. Then since $\mathrm{Fnd}(G)=\emptyset$, there is $H_{1} \in \mathrm{Rem}(G)$ with isomorphism $f_{1} \in \mathrm{Iso}_{G}(H_{1})$ such that $P_{0} \subseteq H_{1}$. Moreover, put $H_{0} = \emptyset$.

Now, suppose we are already decided about $v_{i}$, $H_{i}$, $f_{i}$ and $P_{i-1}$ for $i=1,2, \ldots$. Since $G \setminus H_{i-1}$ is still connected, there is a finite path $P^{*}_{i-1}$ in $G \setminus H_{i-1}$ with endpoints $v_{i}$ and $f_{i}(v_{i})$. Suppose $P^{*}_{i-1}=x_{1}x_{2} \ldots x_{k}$ for some natural $k \geq 2$. So, $x_{1} = v_{i}$ and $x_{k} = f_{i}(v_{i})$. Since $H_i$ and $f_{i}(H_{i})$ are vertex disjoint and $x_{1} \in V(H_{i})$ and $x_{k} \in V(f_{i}(H_{i}))$, it can be inferred that there is $1 < j \leq k$ such that $x_{1}P^{*}_{i}x_{j-1} \subseteq H_{i}$ but $x_{j} \notin V(H_{i})$. So, the edge $x_{j-1}x_{j}$ must be an edge of $\mathrm{link}(H_{i})$.

On the other hand, since $\mathrm{Fnd}(G \setminus H_{i})=\emptyset$, there is $H^{*}_{i} \in \mathrm{Rem}(G \setminus H_{i})$ such that $x_{j} \in V(H^{*}_{i})$. By Proposition \ref{union-rem}, $H_{i+1}= H_{i} \cup H^{*}_{i}$ is a removable subgraph of $G$, and hence, there is an isomorphism $f_{i+1} \in \mathrm{Iso}_{G}(H_{i+1})$. Now, we put $v_{i+1} = x_{j}$ and $P_{i} = P_{i-1} \cup v_{i}P^{*}_{i-1} v_{i+1}$.

It remains to show that $P_{i-1}$ and $v_{i}P^{*}_{i-1} v_{i+1}$ are edge-disjoint. This actually is simple, because if $P_{i-1}=v_{1}y_{2} \ldots y_{t-1}v_{i}$, then $v_{1}P_{i-1}y_{t-1} \subseteq H_{i-1}$ but $v_{i}P^{*}_{i} v_{i+1} \subset G \setminus H_{i-1}$.

Hence, by the above induction, the induced subgraph of $P = \cup_{i \in \mathbb{N} \cup \{ 0 \}} P_{i}$ in $G$ is a ray and we are done with the proof.
\end{proof}

\begin{corollary}
Let $G$ be connected rayless self-contained graph. Then $\mathrm{Fnd}(G) \neq \emptyset$.
\end{corollary}

\begin{corollary}
Let $G$ be self-contained graph such that $G \setminus \mathrm{Fnd}(G)$ is connected. Then, $G$ contains a ray. 
\end{corollary}

\section{Twins of disconnected graphs and disconnected twins of connected graphs}
In this section, we show that the general version of graph alternative conjecture, i. e., Conjecture \ref{general}, is true whenever the connected version, i. e., Conjecture \ref{connected}, is true. In order to do this, we use the following statement whose proof is straightforward.

\begin{proposition}
\label{twin-constrain}
A graph $G$ has a strong twin if and only if $G$ is a self-contained graph which has a non-empty induced subgraph $P$ such that $P \notin \mathrm{Rem}(G)$ but there is an $H \in \mathrm{Rem}(G)$ such that $P \subsetneq H$.
\end{proposition}

We say $G$ has a \emph{strong twin through $H$} if $H \in \mathrm{Rem}(G)$ and there is a non-empty set, namely $P \subsetneq H$, such that $P \notin \mathrm{Rem}(G)$.

\begin{lemma}
\label{x}
Let $G$ be a connected graph which has a disconnected strong twin. Then $G$ has infinitely many strong twins.
\end{lemma}
\begin{proof}
Let $G_1$ have a disconnected strong twin $G'$. Then there is a removable subgraph $H$ and a non-empty set, namely $P \subset H$, such that $P \notin \mathrm{Rem}(G_{1})$ and $G'=G_{1}\setminus P$. Moreover, let $f:G_{1}\longrightarrow G'$ be an arbitrary embedding. Since $G'$ is a strong twin of a connected graph $G_1$, it must have a connected component $\hat{G}$ which contains $f(G_{1})$ as a subgraph. Let $Q$ be an arbitrary connected component of $G'\setminus \hat{G}$ and suppose that $P'$ is union of $P$ and all connected components of $G'$ other than $\hat{G}$ and $Q$. Then, it can be deduced that $G_2 = G_1 \setminus P'$ is also a disconnected strong twin of $G_1$ which has 2 connected components.

To produce infinitely many strong twins for $G_1$, notice that $G_{i}:= G_{1} \setminus \bigcup_{j=0}^{i-1} f^{j}(P')$, for $i=2,3,\ldots$, is also a strong twin of $G_1$ which has $i$ connected components. Therefore, $G_2, G_{3},\ldots$ are all mutually non-isomorphic strong twins of $G_1$.
\end{proof}

Now we can prove Theorem \ref{main}.

\begin{proof}[Proof of Theorem \ref{main}.] Let Conjecture \ref{connected} be true. Suppose that $G$ is a self-contained graph which has a strong twin $G'=G \setminus P$ through $H \in \mathrm{Rem}(G)$, where $P \subsetneq H$, $P \neq \emptyset$ and let $f\in \mathrm{Iso}_{G}(H)$.

First suppose that $G$ is connected. Then either $G'$ is also connected, for which Conjecture \ref{connected} implies that $G$ has infinitely many strong twins, or, $G'$ is disconnected, for which we already know infinitely many strong twins by Lemma \ref{x}. Therefore, hereinafter, we assume that $G$ is disconnected. Moreover, for our convenience, when $G_i$ is a connected graph, by the notation $\mathcal{M}_{G}(G_i)$ we show the set of connected components of $G$ which are isomorphic to $G_i$. And, $\mathcal{N}_{G}(G_i)$ stands for the cardinal number of $\mathcal{M}_{G}(G_i)$.

Now, we split our argument into two cases: whether $P$ contains a connected component of $G$ or not.

\begin{itemize}
\item[Case 1.] $P$ contains no connected component of $G$. Consider  $$\mathcal{A}=\{ C \text{ } \vert \text{ } C \text{ is a connected component of }G\text{ and } C\setminus P \not\simeq C \}.$$ It is obvious that $\mathcal{A}$ is not empty, because else $P$ must be a removable subgraph of $G$, which is a contradiction. Moreover, there exists a $\hat{G}\in \mathcal{A}$ such that either $\mathcal{N}_{G}(\hat{G})\neq \mathcal{N}_{G \setminus P}(\hat{G})$ or there is at least one connected component of $\hat{G} \setminus P$, namely $Q$, such that $\mathcal{N}_{G}(Q)\neq \mathcal{N}_{G \setminus P}(Q)$ (because if for all $\hat{G}\in \mathcal{A}$ and for all connected component $Q$ of $\hat{G} \setminus P$ we have $\mathcal{N}_{G}(\hat{G})= \mathcal{N}_{G \setminus P}(\hat{G})$ and $\mathcal{N}_{G}(Q)= \mathcal{N}_{G \setminus P}(Q)$, then $P$ must be a removable subgraph of $G$, contrary to our assumption that $G\setminus P$ is a strong twin for $G$). Therefore, we have the following two cases:
\begin{itemize}
\item[Case 1.1.] There is at least one connected component of $\hat{G} \setminus P$, namely $Q$, such that $\mathcal{N}_{G}(Q)\neq \mathcal{N}_{G \setminus P}(Q)$.
\begin{itemize}
\item[Case 1.1.1.] We have $f(\hat{G}) \subseteq \hat{G}$. Then, by defining $P':=\hat{G} \cap P$ we have $\hat{G}$ is a connected self-contained graph which has a strong twin $\hat{G} \setminus P'$. Therefore, either $\hat{G} \setminus P'$ is connected, for which Conjecture \ref{connected} implies that $\hat{G}$ has infinitely many strong twins, or, $\hat{G} \setminus P'$ is disconnected, for which we can construct infinitely many strong twins by the method of Lemma \ref{x}. If we change $\hat{G}$ with each of these twins, we obtain infinitely many strong twins for $G$, unless $G$ has infinitely many connected components isomorphic to all but finitely many of them. If this later case happens, choose two of these twins, namely $\hat{G}_{1}$ and $\hat{G}_{2}$, such that $\aleph_0 \leq \mathcal{N}_{G}(\hat{G}_{1}) \leq \mathcal{N}_{G}(\hat{G}_{2})$. Then for each $k\in\mathbb{N}$, $G_k$, which is the resulting graph from removing all but $k$ connected component isomorphic to $\hat{G}_{1}$ from $G$, is a strong twin for $G$. This is true as the following discussion shows. For each $k$, it is evident that $G_k$ has a copy of $G$ in itself because we can produce $\mathcal{N}_{G}(\hat{G}_{1})-k=\mathcal{N}_{G}(\hat{G}_{1})$ connected components isomorphic to $\hat{G}_{1}$ from $\mathcal{M}_{G}(\hat{G}_{2})$ while the number of isomorphic copies of $\hat{G}_{2}$ remain unchanged (because $\mathcal{N}_{G}(\hat{G}_{2})$ is an infinite number greater or equal to $\mathcal{N}_{G}(\hat{G}_{1})$). And, it is also obvious that for each pair of natural numbers $r\neq s$, we have $G_r \not\simeq G_s$.

\item[Case 1.1.2.] We have $f(\hat{G}) \nsubseteq \hat{G}$. 

\begin{itemize}
\item[Case 1.1.2.1.] $\mathcal{N}_{G}(Q) < \infty$. If $\mathcal{N}_{G}(Q) < \mathcal{N}_{G \setminus P}(Q)$, then $H \setminus P$ must have non-empty intersection with some members of $\mathcal{M}_{G \setminus P}(Q)$, because we must have $G\setminus H \simeq G$. For a natural number $s$ such that $1\leq s \leq \mathcal{N}_{G \setminus P}(Q)$, we can choose $s$ of these intersections, namely $H \cap Q_1, H\cap Q_2, \ldots, H\cap Q_s$. Let $$F'= \bigcup_{Q^{*}\in\mathcal{M}_{G \setminus P}(Q)}(H\cap Q^{*}),$$ and $$F=F'\setminus \bigcup_{k=1}^s (H\cap Q_k).$$ Now, let $P'=P\cup F$. It is evident that $P'\subsetneq H$ is still a non-removable subset of $G$. Moreover, we have $\mathcal{N}_{G \setminus P'}(Q)= \mathcal{N}_{G}(Q)+s$. 

Else, if $\mathcal{N}_{G \setminus P}(Q)< \mathcal{N}_{G}(Q)$, then let $t=\mathcal{N}_{G}(Q)-\mathcal{N}_{G \setminus P}(Q)$. Thus, $P$ must have non-empty intersection with some members of $\mathcal{M}_{G}(Q)$, so that removing $P$ changes isomorphism class of those components. Let $F'$ be the set of these intersections, so $\vert F' \vert >t$, because we know that $\hat{G} \setminus P$ has at least one connected component isomorphic with $Q$. Therefore, we must have $1\leq \mathcal{N}_{\hat{G} \setminus P}(Q)=s< t$. Then, let $P'=P \cap \hat{G}$. It is evident that $P'\subsetneq H$ is still a non-removable subset of $G$, and we have $\mathcal{N}_{G \setminus P'}(Q)= \mathcal{N}_{G}(Q)+s$.

Therefore, since $\mathcal{N}_{G \setminus P}(Q)\neq \mathcal{N}_{G}(Q)$, and the fact that one and only one of the above possibilities is true, we have each $G_k = G \setminus \bigcup_{i=0}^{k-1} f^{i}(P')$, for $k=1,2,\ldots$, is a distinct strong twin for $G$. This is because $G_k \setminus \bigcup_{i=0}^{k-1} f^{i}(H)$ is also isomorphic to $G$, and each $G_k$ has exactly $\mathcal{N}_{G}(Q)+ks$ connected components isomorphic to $Q$.

\item[Case 1.1.2.2.] $\mathcal{N}_{G}(Q)$ is an infinite number. Then either $\mathcal{N}_{G \setminus P}(Q)< \mathcal{N}_{G}(Q)$, or, $\mathcal{N}_{G \setminus P}(Q)> \mathcal{N}_{G}(Q)$. First suppose that $\mathcal{N}_{G \setminus P}(Q)< \mathcal{N}_{G}(Q)$. Therefore, if $\alpha$ is the infinite cardinal number of connected components isomorphic with $Q$ after removing $H\setminus P$ from $G\setminus P$, then we must have $\mathcal{N}_{G}(Q)=\mathcal{N}_{G \setminus P}(Q)+\alpha$, because we must have $G\setminus H\simeq G$. This means that $\alpha=\mathcal{N}_{G}(Q)$. Thus, we can choose a countable sequence from $\mathcal{M}_{G}(Q)$, namely, $Q_1,Q_2,\ldots$, and put $$P_k=P\cup\bigcup_{Q^{*}\in \mathcal{M}_{G}(Q)\setminus \{Q_1,\ldots,Q_k\}}Q^*.$$ Moreover, put $$H'=H\cup\bigcup_{Q^{*}\in \mathcal{M}_{G}(Q)}Q^*.$$ We can easily show that $H'$ is a removable subgraph of $G$, and for each $k=1,2,\ldots$, we have $P_k$ is a non-removable subset of $H'$. Moreover each $G_k=G\setminus P_k$ has exactly $k$ connected component isomorphic to $Q$, so they are all mutually non-isomorphic strong twins for $G$.

Then, suppose that $\mathcal{N}_{G \setminus P}(Q)> \mathcal{N}_{G}(Q)$. Therefore, $H\setminus P$ removes an infinite number with cardinality equal to $\mathcal{N}_{G \setminus P}(Q)$ to lessen the number of connected components isomorphic with $Q$ to $\mathcal{N}_{G}(Q)$. Since $\vert \mathcal{M}_{G \setminus P}(Q)\setminus \mathcal{M}_{G}(Q)\vert =\mathcal{N}_{G \setminus P}(Q)$, we can choose a subset $F \subsetneq \mathcal{M}_{G \setminus P}(Q)\setminus \mathcal{M}_{G}(Q)$ such that $\vert F \vert = \mathcal{N}_{G}(Q)$. Now, let $$H'=(H\setminus F) \cup\bigcup_{Q^{*}\in \mathcal{M}_{G}(Q)}Q^{*},$$ and put $$P_k = P\cup\bigcup_{Q^{*}\in \mathcal{M}_{G}(Q)\setminus \{Q_1,\ldots,Q_k\}}Q^*$$ for some definite countable sequence $\{ Q_k \}_{k\in \mathbb{N}}$ of members of $\mathcal{M}_{G}(Q)$. Then $H'$ is a removable subgraph of $G$ and for each $k=1,2,\ldots$, $P_k \subsetneq H'$ is non-removable set such that $G_k=G \setminus P_k$ has exactly $k$ connected component isomorphic to $Q$.
\end{itemize}

\end{itemize}

\item[Case 1.2.] $\mathcal{N}_{G}(\hat{G})\neq \mathcal{N}_{G\setminus P}(\hat{G})$. 

\begin{itemize}
\item[Case 1.2.1.] We have $f(\hat{G}) \subseteq \hat{G}$. Then, by defining $P':=\hat{G} \cap P$ and following just like Case 1.1.1, we can construct infinitely many strong twins for $G$.

\item[Case 1.2.2.] We have $f(\hat{G}) \nsubseteq \hat{G}$. Then, we can split this case into two cases: 
\begin{itemize}
\item[Case 1.2.2.1.] $\mathcal{N}_{G}(\hat{G}) < \infty$. Then, since $P \subsetneq H\in \mathrm{Rem}(G)$, there must be a another connected component of $G$, namely $\hat{G}'$, such that $\hat{G}' \setminus H$ has $s$ connected components isomorphic to $\hat{G}$. Because if there is no such a $G'$, then we cannot have $G \setminus H \simeq G$. The cardinal number $s$ cannot be infinite because else we must have $\mathcal{N}_{G\setminus H}(\hat{G})\geq \aleph_0$, which is a contradiction since $G\setminus H\simeq G$ implies $\mathcal{N}_{G\setminus H}(\hat{G})=\mathcal{N}_{G}(\hat{G})$. Therefore, $s$ is also finite.

Now, by letting $P'=H\cap \hat{G}'$, we have $G_k = G \setminus \bigcup_{i=0}^{k-1} f^{i}(P')$, for $k=1,2,\ldots$, are infinitely many distinct strong twins of $G$. This is because each $G_k$ has $t+ks$ connected components isomorphic to $\hat{G}'$ and $G_k \setminus \bigcup_{i=0}^{k-1} f^{i}(H)$ is also isomorphic to $G$.

\item[Case 1.2.2.2.] $\mathcal{N}_{G}(\hat{G})$ is an infinite number. Then, similar to Case 1.1.2.2, we have either $\mathcal{N}_{G \setminus P}(\hat{G})< \mathcal{N}_{G}(\hat{G})$, or, $\mathcal{N}_{G \setminus P}(\hat{G})> \mathcal{N}_{G}(\hat{G})$. 

If $\mathcal{N}_{G \setminus P}(\hat{G})< \mathcal{N}_{G}(\hat{G})$, then removing $H\setminus P$ must result in producing a set, with cardinality of $\mathcal{N}_{G}(\hat{G})$, consisting of connected components isomorphic with $\hat{G}$. Therefore, if $\alpha$ is the infinite cardinal number of connected components isomorphic with $\hat{G}$ after removing $H\setminus P$ from $G\setminus P$, then we must have $\mathcal{N}_{G}(\hat{G})=\mathcal{N}_{G \setminus P}(\hat{G})+\alpha$. This means that $\alpha=\mathcal{N}_{G}(\hat{G})$. Thus, we can choose a countable sequence from $\mathcal{M}_{G}(\hat{G})$, namely, $\hat{G}_{j_1},\hat{G}_{j_2},\ldots$, and let $$P_k=P\cup\bigcup_{\hat{G}^{*}\in \mathcal{M}_{G}(\hat{G})\setminus \{\hat{G}_{j_1},\ldots,\hat{G}_{j_k} \}}\hat{G}^*.$$ Moreover, put $$H'=H\cup\bigcup_{\hat{G}^{*}\in \mathcal{M}_{G}(\hat{G})}\hat{G}^*.$$ It is easy to see that $G\setminus H'\simeq G\setminus H$, so $H'$ is also a removable subgraph of $G$. Meanwhile, for each $k=1,2,\ldots$ we have $P_k$ is non-removable subset of $H'$. Consequently, each $G_k=G\setminus P_k$ has exactly $k$ connected component isomorphic to $\hat{G}$, so they are all mutually non-isomorphic strong twins for $G$.

Else, if $\mathcal{N}_{G \setminus P}(\hat{G})> \mathcal{N}_{G}(\hat{G})$, then $H\setminus P$ removes an infinite set with the cardinality equal to $\mathcal{N}_{G \setminus P}(\hat{G})$ to lessen the number of connected components isomorphic to $\hat{G}$ to $\mathcal{N}_{G}(\hat{G})$. Since $\vert \mathcal{M}_{G \setminus P}(\hat{G})\setminus \mathcal{M}_{G}(\hat{G})\vert =\mathcal{N}_{G \setminus P}(\hat{G})$, we can choose a subset $F \subsetneq \mathcal{M}_{G \setminus P}(\hat{G})\setminus \mathcal{M}_{G}(\hat{G})$ such that $\vert F \vert = \mathcal{N}_{G}(\hat{G})$. Now, let $$H'=(H\setminus F)\cup\bigcup_{\hat{G}^{*}\in \mathcal{M}_{G}(\hat{G})}\hat{G}^{*},$$ and put $$P_k = P \cup \bigcup_{\hat{G}^{*}\in \mathcal{M}_{G}(\hat{G})\setminus \{\hat{G}_{j_1},\ldots,\hat{G}_{j_k}\}}\hat{G}^*$$ for some definite countable sequence $\{ \hat{G}_{j_k} \}_{k\in \mathbb{N}}$ of members of $\mathcal{M}_{G}(\hat{G})$. Then, $H'$ is a removable subgraph of $G$ and for each $k=1,2,\ldots$, $P_k$ is a non-removable subset of $H'$ such that $G_k=G \setminus P_k$ has exactly $k$ connected component isomorphic to $\hat{G}$.

\end{itemize}

\end{itemize}

\end{itemize}

\item[Case 2.] $P$ contains a connected component $\hat{G}$ of $G$. So, we have the following two possibilities:

\begin{itemize}
\item[Case 2.1.] $\mathcal{N}_{G \setminus P}(\hat{G})= \mathcal{N}_{G}(\hat{G})$. Let 

\begin{multline*} \mathcal{B}=\{ C^\circ  \vert C \textsl{ is a connected component of }G, C\subseteq P\textsl{ and} \\ \mathcal{N}_{G \setminus P}(C)= \mathcal{N}_{G}(C) \},
\end{multline*}
where $C^\circ$ stands for a representative for the isomorphism class of $C$, i. e., $\mathcal{B}$ contains only one component form those connected components of $G$ which are also contained in $P$, they have the same cardinal number in $G$ and $G\setminus P$ and belong to the same isomorphic class. Then, by defining $$P'=\big( P \cup \bigcup_{C \in \mathcal{B}} \mathcal{M}_{G \setminus P}(C)\big) \setminus \bigcup_{C \in \mathcal{B}} \mathcal{M}_{G}(C),$$ and $$ H'=\big( H \cup \bigcup_{C \in \mathcal{B}} \mathcal{M}_{G \setminus P}(C)\big) \setminus \bigcup_{C \in \mathcal{B}} \mathcal{M}_{G}(C),$$ we have $P'\subsetneq H'\in\mathrm{Rem}(G)$ and $G\setminus P'\simeq G \setminus P$. Moreover, $P'$ contains no connected component of $G$, namely $C$, such that $\mathcal{N}_{G \setminus P}(G_j)= \mathcal{N}_{G}(G_j)$.  Therefore, by Case 1 or Case 2.2, $G$ has infinitely many strong twins through $H'$.

\item[Case 2.2.] $\mathcal{N}_{G \setminus P}(\hat{G})\neq\mathcal{N}_{G}(\hat{G})$.
\begin{itemize}

\item[Case 2.2.1.] $\mathcal{N}_{G}(\hat{G})=t < \infty$. Then, by similar wording like Case 1.2.2.1, we can construct infinitely many strong twins for $G$.

\item[Case 2.2.2.] $\mathcal{N}_{G}(\hat{G})$ is an infinite number. Then just like Case 1.2.2.2, we are able to construct infinitely many strong twins for $G$.
\end{itemize}
 
\end{itemize} 
\end{itemize}

\end{proof}

It is worth mentioning that the converse of Theorem \ref{main} is not evident because a counterexample to Conjecture \ref{connected} can still have infinitely many disconnected strong twins, i. e., it might not be a counterexample to Conjecture \ref{general}.

\section{Torsion of a removable subgraph}

In this section, we try to find an answer to the following question: let $P$ and $Q$ be removable subgraphs of a self-contained graph $G$. Under what conditions on $P$ and $Q$, we have $P \cup Q$ is another removable graph of $G$? Before answering to this, we show that there are self-contained graphs that for some removable subgraphs of them like $H$, the foundation does not remain invariant after removal of $H$. This fact uncover another substructure of self-contained graphs which has an impact on the answer to the our question.

In the following proposition we use the notation $\mathrm{Fin}(H)$ for the set $\{ a \in V(H) : \mathrm{deg}_G (a) < \infty \}$ where $H$ is an induced subgraph of $G$.

\begin{proposition}
\label{14} Let $G$ be a self-contained graph, $v \in V( \mathrm{Fnd}(G))$ and $$\Big| \mathrm{Fin}\big( \mathrm{Fnd}(G) \big) \Big| < \infty .$$ If $\mathrm{deg}_{G}(v) < \infty$ then $N_{G}(v) \subseteq V(\mathrm{Fnd}(G))$.
\end{proposition}
\begin{proof} If $N_{G}(v) \nsubseteq V(\mathrm{Fnd}(G))$, then $v$ is adjacent to one vertex $b$ which is not a vertex of the foundation. Hence, $b$ is a vertex of a removable subgraph $H$ whose elimination from $G$ reduces the degree of $v$. Now let $f: G \longrightarrow G \setminus H$ be an isomorphism. By Proposition \ref{12}, $f$ induces an injection on $\mathrm{Fnd}(G)$. Hence, $v$, $f(v)$, $f^2 (v)=f(f(v))$, ... are all distinct vertices of  $\mathrm{Fnd}(G) \subseteq f \big( \mathrm{Fnd}(G) \big) \subseteq f^{2}\big( \mathrm{Fnd}(G) \big) \ldots$ respectively. Therefore, for each $n \in \mathbb{N}$, $f^{n} \big( \mathrm{Fnd}(G) \big)$ is isomorphic to $\mathrm{Fnd}(G)$ and contains $n$ distinct finite degree vertices. So, by common induction, $\mathrm{Fnd}(G)$ contains infinitely many finite-degree vertices, contradiction to our assumption.
\end{proof}
\begin{corollary}
Let $G$ be a connected self-contained graph and $\mathrm{Fnd}(G) \neq \emptyset$. Then, if $\mathrm{Fnd}(G)$ is a finite graph, then it contains a vertex $v$ which has an infinite degree in $G$.
\end{corollary}
\begin{example}
\label{15} The condition $\Big| \mathrm{Fin}\big( \mathrm{Fnd}(G) \big) \Big| < \infty$ in Proposition \ref{14} is not redundant.  Let $A= \{ a_{1}, a_{2}, \ldots \}$, $B= \{ b_{1}, b_{2}, \ldots \}$ and $C= \{ c_{0}, c_{-1}, c_{-2} \ldots \}$ be three sets equipotent to natural numbers disjoint from $\mathbb{Z}$ and from each other, and suppose that $G$ is a graph with vertex set $V(G) = \mathbb{Z} \cup A \cup B \cup C$ and edge set of the form of union of the following four sets:
\begin{itemize}
\item[i] $\big\{ \{ v_{1}v_{2} \} : v_{1}, v_{2} \in \mathbb{Z}$ and $\vert v_{1} - v_{2} \vert  = 1 \big\}$
\item[ii] $\big\{ \{ a_{j}j \} : j \in \mathbb{N} \big\}$
\item[iii] $\big\{ \{ b_{j}j \} : j \in \mathbb{N} \big\}$
\item[iv] $\big\{ \{ c_{j}j \} : j \in \mathbb{Z}$ and $j\leq 0 \big\}$.
\end{itemize}
Then $G$ is a self-contained graph since $P =\{ a_{1} \}$ and $Q= \{ b_{1} \}$ are two removable subgraph of $G$. Moreover, $\mathrm{Fnd}(G) = \mathbb{Z} \cup C$ and for every $j > 0$, $j$ is a finite-degree vertex of the foundation of $G$ which has a neighbor out of $\mathrm{Fnd}(G)$.
\end{example}

Example \ref{15} reveals another important substructure of self-contained graphs. Let $G$ be a self-contained graph and $H \in \mathrm{Rem}(G)$. A vertex $v$ of $G$ is called a \emph{twisted vertex} for $H$ if there exists $P \in \mathrm{Rem}(G)$ such that $v \in V(P)$ and $v \in \mathrm{Fnd}(G \setminus H)$. The subgraph induced by all twisted vertices for $H$ is called the \emph{torsion} of $H$ and is denoted by $\mathrm{Tor}_{G} (H)$.  Meanwhile, when $\mathrm{Tor}_{G} (H) = \emptyset$ we say $H$ is a \emph{torsion-free removable subgraph} of $G$. Moreover, we say $G$ is \emph{torsion-free self-contained graph} if all removable subgraphs of it are torsion-free.

Torsion of a removable subgraph can be the null graph, as it is in all examples prior to Example \ref{15}. But in Example \ref{15}, if we consider $H=\{ a_{1} \}$ as a removable subgraph of $G$, we have $ \mathrm{Tor}_{G} (H) = \{ b_{1} \}$. By the way, it is obvious from previous definition that for a self-contained graph $G$, if $\mathrm{Fnd}(G) = \emptyset$, then for all removable subgraph $H$ of $G$ we have $\mathrm{Tor}_{G} (H) = \emptyset$, i.e. $G$ is torsion-free.

The following two propositions express some properties of torsions of removable subgraphs:
\begin{proposition}
\label{tensub}
Let $G$ be a self-contained graph, $P,Q \in \mathrm{Rem}(G)$ and $P \subseteq Q$. Then $\mathrm{Tor}_{G} (P)\subseteq \mathrm{Tor}_{G} (Q)$.
\end{proposition}
\begin{proof}
If $P=Q$, there is nothing to prove. Hence, suppose $Q \setminus P \neq \emptyset$. Then $Q \setminus P \in \mathrm{Rem}(G \setminus P)$. Now, if $a \in V \big( \mathrm{Tor}_{G} (P) \big)$, then $a \in V \big( \mathrm{Fnd}(G \setminus P) \big) \subseteq V \Big( \mathrm{Fnd} \big( (G \setminus P) \setminus (Q \setminus P) \big) \Big)$ which means that $a \in \mathrm{Fnd}(G \setminus Q)$. We have $a \notin Q$ since $a \notin P$ and $a \notin Q \setminus P$ because $Q \setminus P \in \mathrm{Rem}(G \setminus P)$ cannot contain a vertex of the foundation of $ G \setminus P$. Therefore, $a$ is a vertex of a removable subgraph of $G$ (since $a \in  V \big( \mathrm{Tor}_{G} (P) \big)$) which is also a vertex of the foundation of $G \setminus Q$, and consequently, $a \in  V \big( \mathrm{Tor}_{G} (Q) \big)$.
\end{proof}

\begin{proposition}
Let $G$ be a self-contained graph, $P,Q \in \mathrm{Rem}(G)$ and $Q \cap \mathrm{Tor}_{G} (P) \neq \emptyset$. Then $P \cap \mathrm{Tor}_{G} (Q) \neq \emptyset$.
\end{proposition}
\begin{proof}
First, we prove that there does not exist an $H \in \mathrm{Rem}(G \setminus Q)$ such that $P \subseteq H$. If on the contrary $P \subseteq H \in \mathrm{Rem}(G \setminus Q)$ then $Q \cup H \in \mathrm{Rem}(G)$ and by Proposition \ref{tensub} we have $\mathrm{Tor}_{G} (P)\subseteq \mathrm{Tor}_{G} (Q \cup H)$; which is a contradiction since $Q \cap \mathrm{Tor}_{G} (P) \neq \emptyset$ but no vertex of $\mathrm{Tor}_{G} (Q \cup H)$ is a vertex of $Q$.

Therefore, since no removable subgraph of $G\setminus Q$ contains all vertices of $P$, there exist an $a \in V(P)$ which a vertex of $\mathrm{Fnd}(G \setminus Q)$. So, $P \cap \mathrm{Tor}_{G} (Q) \neq \emptyset$.
\end{proof}

\begin{theorem}
Let $G$ be a self-contained graph, $H \in \mathrm{Rem}(G)$ and $\mathrm{Tor}_{G} (H) \neq \emptyset$. Then $\mathrm{Fnd}(G)$ is a self-contained graph which has a removable subgraph $P$ isomorphic to $\mathrm{Tor}_{G} (H)$.
\end{theorem}
\begin{proof}
Let $f: G \longrightarrow G \setminus H$ be an isomorphism. Hence, we have $\mathrm{Fnd}(G) \cong f \big( \mathrm{Fnd}(G) \big) = G \big[ V(\mathrm{Fnd}(G)) \cup V(\mathrm{Tor}_{G} (H)) \big]$, which means that $f \big( \mathrm{Fnd}(G) \big)$ is a self-contained graph and $\mathrm{Tor}_{G} (H) \in \mathrm{Rem} \Big( f \big( \mathrm{Fnd}(G) \big) \Big)$. Therefore, $\mathrm{Fnd}(G)$ is isomorphic to a self-contained graph that has a removable subgraph $P=f^{-1} \big( \mathrm{Tor}_{G}(H) \big)$ which is isomorphic to $\mathrm{Tor}_{G}(H)$.
\end{proof}

Even when $G$ is a torsion-free self-contained graph and $P, Q \in \mathrm{Rem}(G)$, we cannot say for sure that $P \cup Q \in \mathrm{Rem}(G)$, because they may have non-empty intersection.

\begin{example}
\label{rhomb-star}
Let $G$ be the infinite graph with the vertex set $\mathbb{N} \cup \{ 0 \}$ and its edges are of the following:
\begin{itemize}
\item $1$ is only adjacent to $0$, but $0$ is adjacent to all natural $n \equiv \pm 1$ (mod 3),
\item If $n$ is a natural number that $n \equiv 0$ (mod 3), then $n$ is adjacent to $n-1$ and $n+1$. 
\end{itemize}
Then $G$ is a torsion-free self-contained graph which has $P = G[ \{2,3,4 \} ]$, $Q = G[ \{ 1, 2,3 \} ]$ and $R = G[ \{ 1, 5,6 \} ]$ as some of its removable subgraphs. While $P \cup R$ is a removable subgraph of $G$, the other two union subgraphs, $P\cup Q$ and $Q \cup R$, are not.
\end{example}

\begin{proposition}
\label{iso-union}
Let $G$ be a self-contained graph and $P, Q \in \mathrm{Rem}(G)$ such that $P \cap Q = \emptyset$. Then $P \cup Q \in \mathrm{Rem}(G)$ if and only if there exists an isomorphism $f: G \longrightarrow G \setminus P$ such that $f^{-1}(Q) \in \mathrm{Rem}(G)$. In particular, if there is an isomorphism $f: G \longrightarrow G \setminus P$ such that $f(Q)=Q$, then $P \cup Q \in \mathrm{Rem}(G)$.
\end{proposition}
\begin{proof}
If $f^{-1}(Q) \in \mathrm{Rem}(G)$, then $Q \in \mathrm{Rem}(G \setminus P)$ and thus by Proposition \ref{union-rem} we have $P \cup Q \in \mathrm{Rem}(G)$. Meanwhile, when $P \cup Q \in \mathrm{Rem}(G)$, we must have $Q \in \mathrm{Rem}(G \setminus P)$ and hence there is an isomorphism $f: G \longrightarrow G \setminus P$ such that $f^{-1}(Q) \in \mathrm{Rem}(G)$.
\end{proof}

Now, one may think that for a self-contained graph $G$ and $P, Q \in \mathrm{Rem}(G)$, if $Q \cap \big( P \cup \mathrm{Tor}_{G}(P) \big) = \emptyset$ then $P \cup Q \in \mathrm{Rem}(G)$. As we show in the following example, sometimes it is not true.
\begin{example}
\label{roller-brush}
Let $R$ be a graph consisting of infinitely many disjoint copies of the graph $G$ we have introduced in Example \ref{15}, i.e. $G_1$, $G_2$, ... and a vertex $a$ which is adjacent to all other vertices of double rays of copies of $G$. Then $R$ is a self-contained graph since every $G_i$ is removable subgraph of $R$ for $i=1,2, \ldots$. Moreover, if $P_i$ and $Q_i$ are copies of removable subgraphs of $G_i$ we have introduced in Example \ref{15}, then they are also removable subgraphs of $R$. but we have $\mathrm{Tor}_{G_i} (P_{i}) = Q_i$ while $\mathrm{Tor}_{R} (P_{i}) = \emptyset$. So, although $Q_{i} \cap \big( P_{i} \cup \mathrm{Tor}_{R}(P_{i}) \big) = \emptyset$, we have $P_{i} \cup Q_{i} \notin \mathrm{Rem}(R)$.
\end{example}

Let $G$ be a self contained graph and $P \in \mathrm{Rem}(G)$. A vertex $v$ of $G$ is a \emph{curly vertex to $P$} if there exists a self-contained removable subgraph $Q \in \mathrm{Rem}(G)$ such that $P \in \mathrm{Rem}(Q)$ and $v \in V \big( \mathrm{Tor}_{Q} (P) \big)$. The set of all curly vertices to $P$ is shown by $\mathrm{Curl}_{G} (P)$.

We end this section with the following conjecture:
\begin{conjecture}
\label{Conj2}
Let $G$ be a self-contained graph and $P, Q \in \mathrm{Rem}(G)$. If $Q \cap \big( P \cup \mathrm{Tor}_{G}(P) \cup \mathrm{Curl}_{G} (P) \big) = \emptyset$ then $Q \in \mathrm{Rem}(G \setminus P)$. 
\end{conjecture}

\section*{Acknowledgments}

The authors owe a great debt to the referee, who have carefully read an earlier version of this paper and made significant notes for improvement. We would like to express our deep appreciation for the referee's work.

\end{document}